\documentclass{article}
\usepackage{amsmath}
\usepackage{amstext}
\usepackage{amssymb}
\usepackage{bbm}
\usepackage{latexsym}
\usepackage{amsthm}
\usepackage{bbm}
\usepackage{enumerate}
\usepackage{geometry}
\usepackage{mathabx}
\usepackage{color}
\usepackage{authblk}
\usepackage[utf8]{inputenc}
\usepackage[OT4]{fontenc}

\def\ttdwa{{\mathbb T}^2}

\def\cc{{\mathbb C}}
\def\zz{{\mathbb Z}}

\newtheorem{thm}{Theorem}

\newtheorem{deff}[thm]{Definition}
\newtheorem{cor}[thm]{Corollary}
\newtheorem{prop}[thm]{Proposition}
\theoremstyle{definition}\newtheorem{Rem}[thm]{Remark}

\title{On the analytic version of  the Mitiagin - DeLeeuw - Mirkhil non-inequality on bi-disc}
\author[1,2]{Krystian Kazaniecki}
\author[3]{Micha{\l} Wojciechowski}
\affil[1]{Institute of Analysis, Johannes Kepler University Linz}
\affil[2]{Institute of Mathematics, University of Warsaw}
\affil[3]{Institute of Mathematics, Polish Academy of Sciences}
\begin{document}
\maketitle
\begin{abstract}
	Using the method of Rudin-Shapiro polynomials we prove the analytic version of the Mitiagin - DeLeeuw - Mirkhil non-inequality for complex partial differential operators with constant coefficients on bi-disc.
\end{abstract}

\vskip5mm
\let\thefootnote\relax\footnotetext{This research was partially supported by the National Science Centre, Poland, and Austrian Science Foundation FWF joint CEUS programme. National Science Centre project no. 2020/02/Y/ST1/00072 and FWF project no. I5231.}
\section{Introduction}
The classical results of Mitiagin, Mirkhil and DeLeeuw and Ornstein (\cite{Leeuw1964, MR0100216, MR0149331}, see also \cite{C4,KSW,KK6}) say that 
for $p=1$ or $p=\infty$, and homogeneous differential operators of the same degree with constant coefficients $P_0(D), P_1(D),\dots , P_k(D)$ the inequality
\begin{equation}\label{jedyny}
\|P_0(D)f\|_p\lesssim \sum_{i=1}^k \|P_i(D)f\|_p
\end{equation}
holds for every smooth function $f$ with bounded support iff 
\[
P_0\in \operatorname{span}\{P_1,\dots, P_k\}.
\]
Similar problem for differential operators with variable coefficients was considered in \cite{MR2451367}. Analogous characterization holds for functions on tori. In the present paper we deal with the simplest case $d=2$, $P_0(x, y)=xy,$  $P_1(x,y)=x^2$ and $P_2(x,y)=y^2$ with additional restriction that $f$ is analytic, i.e. $\widehat{f}(n)=0$ for $n$ outside the set $\{n\in\mathbb{Z}^2; n_1\geq 0,\; n_2\geq 0\}$. Under such restriction for $p=1$ the inequality \eqref{jedyny} holds true. Indeed $D^2_{xy}f= T_m(D^2_{x^2}f+D^2_{y^2}f)$ where $T_m$ is an invariant operator given by the multiplier $m(n_1,n_2)=\frac{n_1 n_2}{n_1^2+n_2^2}$. By (\cite{MW}, see also \cite{rFef}, \cite{Carl} ) it is bounded in $L^1$ norm for analytic functions.
Surprisingly, in this note we show that the situation in $L^\infty$ norm is different. Namely we have:
\begin{thm}\label{Torbezs}
For every $C>0$, there exists an analytic trigonometric polynomial $f$ on the torus $\ttdwa$ such that
\[ \|D^2_{xx}f\|_\infty+\|D^2_{yy}f\|_\infty\leq 1\] but \[\|D^2_{xy}f\|_\infty>C.\]
\end{thm}

It is not difficult to see that the Poisson extensions
of $D^2_{xx}f$, $D^2_{yy}f$ and $D^2_{xy}f$ onto unit bi-disc are respectively $z_2\frac{\partial}{\partial z_2}\big(z_1\frac{\partial}{\partial z_1}\big)F$,
$\big(z_1\frac{\partial}{\partial z_1}\big)^2F$ and 
$\big(z_2\frac{\partial}{\partial z_2}\big)^2F$, where $F$ is the Poisson extension of $f$. Using the maximum principle we can then reformulate Theorem \ref{Torbezs}:

\begin{thm}\label{dyskbezs}
For every $c>0$, there exists an analytic polynomial $F$ on the unit bi-disc such that
\[
\left\|\left(z_1\frac{\partial}{\partial z_1}\right)^2F\right\|_{A(D^2)}\leq 1 \mbox{ and }
\left\|\left(z_2\frac{\partial}{\partial z_2}\right)^2F\right\|_{A(D^2)}\leq 1 
\]
but
\[
\left\|z_2\frac{\partial}{\partial z_2}\left(z_1\frac{\partial}{\partial z_1}\right)F\right\|_{A(D^2)}\geq c.
\]
\end{thm}

Moreover, we provide the dependence of the above constant $c$ on the degree of $F$ (Theorem \ref{thm: stala} and Corollary \ref{stawka}). A double logarithmic bound on the growth of the constant in Theorem \ref{Torbezs} (without analyticity) could be derived from the construction of Curc{\u a} (cf.\cite{curca} sections 5.1 and  5.2;). Our construction not only gives analytic example but also provides (almost) logarithmic bound.
\begin{thm}\label{thm: stala}
There exists an analytic polynomial $P_n$ of degree $n$ of two complex variables such that 
\[
\|\big(z_1\frac{\partial}{\partial z_1}\big)^2P_n\|_{A(D^2)}\leq 1\quad\mbox{ and }\quad
\|\big(z_2\frac{\partial}{\partial z_2}\big)^2P_n\|_{A(D^2)}\leq 1
\]
but
\[\|z_2\frac{\partial}{\partial z_2}\big(z_1\frac{\partial}{\partial z_1}\big)P_n\|_{A(D^2)}\gtrsim \log^{1/8} n.
\]
\end{thm}
The similar quantitative result follows immediately on tori for boundary values of analytic functions.  
\begin{cor}\label{stawka}
There exists an analytic trigonometric polynomial $P_n$ of degree $n$ of two complex variables such that 
\[ \|D^2_{xx}f\|_\infty+\|D^2_{yy}f\|_\infty\leq 1\] but \[\|D^2_{xy}f\|_\infty \gtrsim \log^{1/8} n.\]
\end{cor}
 For some other problems related to quantitative estimates of the norm of derivatives from below see \cite{Tely}.
 
With a little effort we can replace the Euler derivatives from Theorem \ref{dyskbezs} by the ordinary derivatives (lower order derivatives appear here to compensate for the lack of compactly embedded support):

\begin{thm}\label{czystebezs}
For every $c>0$, there exists an analytic polynomial $F$ on unit bi-disc such that
\[
\big\|F\big\|_{A(D^2)}\leq 1,\;\; \left\|\frac{\partial}{\partial z_1}F\right\|_{A(D^2)}\leq 1,\;\;\left\|\frac{\partial}{\partial z_2}F\right\|_{A(D^2)}\leq 1,
\]
and
\[
 \left\|\frac{\partial^2}{\partial z^2_1}F\right\|_{A(D^2)}\leq 1,\;\; \left\|\frac{\partial^2}{\partial z^2_2}F\right\|_{A(D^2)}\leq 1
\]
but
\[
\left\|\frac{\partial^2}{\partial z_1 \partial z_2} F\right\|_{A(D^2)}\geq c.
\]
\end{thm}
Despite we are mostly interested in the quantitative estimates, in the last section of the article we provide another purely qualitative proof.

\begin{Rem}
Our qualitative method (Theorem 1, Section \ref{sect: qapp}) extends to the general situation i.e. homogeneous differential operators of the same degree with constant coefficients $P_0(D), P_1(D),\dots , P_k(D)$ such that $P_0\notin \operatorname{span}\{P_1,\dots, P_k\}$. However when it comes to quantitative results (Theorem \ref{thm: stala}, Corollary \ref{stawka}) we only know how to extend them under some geometrical conditions as in \cite[Theorem 2]{KW}. 
\end{Rem}

Motivated by Theorem \ref{czystebezs} we propose the following definition.

\begin{deff}
Let $\Omega\subset\cc^2$ and $z\in\partial \Omega$ be its boundary point. We say that
$z$ is a local MDM point for $\Omega$ provided there exists a ball $B$ centered at $z$ such that for every $c>0$, there exists an analytic polynomial $F$ on $\Omega\cap B$ such that
\[
\big\|F\big\|_{A(\Omega\cap B)}\leq 1,\;\; \left\|\frac{\partial}{\partial z_1}F\right\|_{A(\Omega\cap B)}\leq 1,\;\; \left\|\frac{\partial}{\partial z_2}F\right\|_{A(\Omega\cap B)}\leq 1
\]
and
\[
\left\|\frac{\partial^2}{\partial z^2_1}F\right\|_{A(\Omega\cap B)}\leq 1,\;\; \left\|\frac{\partial^2}{\partial z_2^2}F\right\|_{A(\Omega\cap B)}\leq 1
\]
but
\[
\big|\frac{\partial^2}{\partial z_1 \partial z_2} F(z)\big|\geq c.
\]
\end{deff}
Similarly for bounded domains $\Omega$ we define the global MDM point by replacing $B$ by $\cc^2$ in the above definition. Using this notion we can reformulate Theorem \ref{czystebezs} saying that point ${\bf 1}:=(1,1)$ is a (global) MDM point of the unit bi-disc (actually it is easy to see that any Shilov boundary point of a bi-disc is its MDM point).
Note that MDM property is inherited by inclusion i.e. if $z\in \overline{\Omega}\subset \overline{\Omega'}$ is an MDM point of $\Omega'$ then $z$ is an MDM point of $\Omega$. There are some natural restrictions on $z\in\partial\Omega$ to be an MDM point. Indeed, suppose $z\in\partial \Omega$ is an internal point of $\hat{\Omega}$ - the polynomial convex hull of $\overline{\Omega}$. Then there exists  a ball $B$ centered at $z$ contained in $\hat{\Omega}$. For every analytic polynomial $f$ we have 
\[
\big|\frac{\partial^2}{\partial z_1 \partial z_2} f(z)\big| \leq C_B \sup\limits_{\omega\in B} |f(\omega)|\leq
 C_B \sup\limits_{\omega\in \Omega} |f(\omega)|.
\]
The first inequality is a consequence of the Cauchy formula \cite[1.2.2. (v)]{Rudin} and second follows by definition of polynomial convex hull \cite[p.~53]{horm}. We conjecture that above obstacle is the only one which prevents the MDM property. However even in the case of the unit ball of $\cc^2$ we do not know whether any MDM point exists.

\vspace{5mm}
\textbf{Acknowledgments} We would like to thank Paul F.X. Müller and Fedor Nazarov for valuable comments and suggestions.

\section{Proofs of quantitative results}
\begin{proof}[Proof of Theorem \ref{dyskbezs}]
Let $(n_k)\subset \zz^2_+$ be a lacunary sequence which will be defined later. Let $(a_k)_{k=1}^n$ be a sequence of scalars. Define the sequence of Rudin - Shapiro polynomials:
\[
p_0(z)= a_0 z^{n_0}\quad\quad q_0=1
\]
and for $n>0$
\[
\aligned
p_n(z)&=p_{n-1}(z)+a_n z^{n_k} { q_{n-1}(z^{-1})} \\
q_n(z)&=q_{n-1}(z)- \bar a_n {z}^{n_k} {p_{n-1}(z^{-1})}
\endaligned
\]
(here we adopt the notation $z^{-1}=(z_1^{-1},z_2^{-1})$ for $z=(z_1,z_2)\in\mathbb C^2$).
By the simple induction, $p_k$'s and $q_k$'s are holomorphic polynomials.
By the elementary identity valid for any $x,y,a\in\mathbb C$,
\[
|x+ay|^2+|y-\bar{a}x|^2=(1+|a|^2)(|x|^2+|y|^2)
\]
we get for any $|z_1|=|z_2|=1$,
\[
|p_n(z)|^2+|q_n(\bar z)|^2=\prod_{k=0}^n(1+|a_k|^2)
\]
Indeed, this follows inductively since for $|z_1|=|z_2|=1$,
\[
\aligned
|p_k(z)|^2+|q_k(\bar z)|^2
 &=|p_{k-1}(z)+a_n z^{n_k} { q_{k-1}(z^{-1})}|^2\\
&\qquad+|q_{k-1}(\bar z)- \bar a_n {\bar z}^{n_k} {p_{k-1}(\bar z^{-1})}|^2\\
&=|p_{k-1}(z)+a_k z^{n_k} { q_{k-1}(z^{-1})}|^2\\
&\qquad+|z^{n_k}q_{k-1}(z^{-1})- \bar a_k {p_{k-1}(z)}|^2\\
&= (1+|a_k|^2)(|p_k(z)|^2+|q_k(\bar z)|^2)
\endaligned
\]
By the construction we get
\[
p_n(z)= a_0z^{n_0}+\sum_{k=1}^n a_k z^{n_k} {q_{k-1}(z^{-1})}
\]
and
\[
q_n(z)=1-\sum_{k=0}^{n} \bar a_k{z}^{n_k}{p_{k-1}(z^{-1})}
\]
For any sequence of complex numbers $\{\omega_k\}_{k=1}^{n}$ there exists a subset $A\subset \{0,1,2,\dots,n\}$ such that
\[
|\sum_{k\in A} \omega_k|\geq \frac{1}{\pi} \sum_{k=1}^{n} |\omega_n|
\]
Note that the values  $p_k({\bf 1})$ and $q_k({\bf 1})$ are independent on the choice sequence $(n_k)$. They only depend on the choice of the sequence $(a_k)_{k=1}^{n}$. For fixed sequence $(a_k)$ we choose larger of two sums $ \sum_{k=0}^n |a_k{q_{k-1}({\bf 1})}|$, $\sum_{k=1}^n | \bar a_k{p_{k-1}({\bf 1})}|$. Without loss of generality we assume it is the first one and we put $\omega_k =a_k{q_{k-1}({\bf 1})}$. Therefore there exists $A:=A((a_k))\subset \{0,1,2,\dots,n\}$ such that
\[
\aligned
|\sum_{k\in A} a_k q_{k-1}&({\bf 1})|+
|\sum_{k\in A}  \bar a_k{p_{k-1}({\bf 1})}|\\
&\geq \frac{1}{2\pi} \big(\sum_{k=0}^n |a_k{q_{k-1}({\bf 1})}|+
\sum_{k=1}^n | \bar a_k{p_{k-1}({\bf 1})}|\big)\\
&=\frac{1}{2\pi} \sum_{k=0}^n |a_k|\big(|{q_{k-1}({\bf 1})}|+|{p_{k-1}({\bf 1})}|\big)\\
&\geq \frac{1}{2\pi}\sum_{k=0}^n |a_k|\big(|{q_{k-1}({\bf 1})}|^2+|{p_{k-1}({\bf 1})}|^2\big)^{\frac{1}{2}}\\
&\geq \frac{1}{2\pi} \sum_{k=1}^n |a_k|\prod_{j=0}^{k-1}(1+|a_j|^2)^{1/2}\\
\endaligned
\]
On the other hand, for $|z_1|=|z_2|=1$ 
\[
\aligned
|\sum_{k\in A} a_k z^{n_k}& {q_{k-1}(z^{-1})}|+
|\sum_{k\in A}  \bar a_k{\bar z}^{n_k}{p_{k-1}(\bar z^{-1})}|\\
&\leq\sum_{k=0}^n |a_k|\cdot |{q_{k-1}(z^{-1})}|+
\sum_{k=0}^n |\bar a_k|\cdot |{p_{k-1}(\bar z^{-1})}|\\
&\leq\sum_{k=0}^n |a_k|\cdot \big(|{q_{k-1}(z^{-1})}|+|{p_{k-1}(\bar z^{-1})}|\big)\\
&\leq \sqrt{2} \sum_{k=1}^n |a_k|\prod_{j=0}^{n}(1+|a_j|^2)^{1/2}\\
\endaligned
\]
Put now $\big(z_1\frac{\partial}{\partial z_1}\big)^2F=p_n$ and $\big(z_1\frac{\partial}{\partial z_1}\big)^2G=q_n-1$. Clearly
\begin{equation}\label{czyste1}
\begin{split}
    \max\{\|\big(z_1\frac{\partial}{\partial z_1}\big)^2F\|_{A(D^2)},\|\big(z_1\frac{\partial}{\partial z_1}\big)^2G\|_{A(D^2)}\}&\leq \max\{\|p_n\|_{A(D^2)},\|q_n\|_{A(D^2)}\}+1
    \\&\leq \left(\sup_{z\in D^2}\lvert p_n(z)\rvert^2+\lvert q_n(z)\rvert^2\right)^{\frac{1}{2}}+1
    \\&= \left(\sup_{\lvert z_1\rvert=\lvert z_2\rvert=1} \lvert p_n(z_1,z_2)\rvert^2+\lvert q_n(z_1,z_2)\rvert^2\right)^{\frac{1}{2}}+1
    \\&\leq \prod\limits_{n=1}^{\infty} (1+a_n^2)^{\frac{1}{2}}+1.
\end{split}
\end{equation}
Function $p_k$ is a polynomial which is a sum of $2^k$ monomials of the form $b_Mz^{m(M)}$, where 
$M=\{i_0,i_1,\dots,i_{2s}\}\subset\{0,1,2,\dots,n\}$ and 
$m(M)=\sum_{j=0}^{2s}(-1)^jn_{i_j}$ and $|b_M|=\prod_{j\in M} |a_{n_j}|$. Similarly $q_k$ is a polynomial which is a sum of 1 and $2^k-1$ monomials of the form $b_Mz^{m(M)}$, where 
$M=\{i_0,i_1,\dots,i_{2s+1}\}\subset\{0,1,2,\dots,n\}$ and 
$m(M)=\sum_{j=0}^{2s+1}(-1)^{j+1}n_{i_j}$ and $|b_M|=\prod_{j\in M} |a_{n_j}|$.
Therefore
\[
z_2\frac{\partial}{\partial z_2}\big(z_1\frac{\partial}{\partial z_1}\big)F(z)=
\sum_{2\notdivides \#M} b_M\frac{m(M)_2}{m(M)_1}z^{m(M)}
\]
and
\[
\big(z_2\frac{\partial}{\partial z_2}\big)^2F(z)=
\sum_{2\notdivides\#M} b_M\left(\frac{m(M)_2}{m(M)_1}\right)^2z^{m(M)}
\]
Similarly
\[
z_2\frac{\partial}{\partial z_2}\big(z_1\frac{\partial}{\partial z_1}\big)G(z)=
\sum_{2|\#M}  b_M\frac{m(M)_2}{m(M)_1} z^{m(M)}
\]
and
\[
\big(z_2\frac{\partial}{\partial z_2}\big)^2 G(z)=
\sum_{2|\#M} b_M\left(\frac{m(M)_2}{m(M)_1}\right)^2 z^{m(M)}
\]
Suppose that the sequence $(n_k)$ was chosen in such a way that for $0<b<1$
\[
\aligned
(i)&\qquad\sum_{\max M\in A}\big|\frac{m(M)_2}{m(M)_1}-\kappa\big|< 1\\
(ii)&\qquad\sum_{\max M\in A}\big|\left(\frac{m(M)_2}{m(M)_1}\right)^2-
\kappa^2\big|<1\\
(iii)&\qquad\sum_{\max M\notin A}\big|\frac{m(M)_2}{m(M)_1}\big|< 1\\
(iv)&\qquad m(M)\in \zz^2_+\qquad\text{for $M\subset\{0,1,2,\dots,n\}$}
\endaligned
\]
The above conditions are satisfied if the ratio $\frac{(n_{k})_2}
{(n_{k})_1}$ is suitably chosen with respect to the fact $k\in A$ or not, and  $\frac{|n_{k+1}|}
{|n_{k}|}$ is big enough. Simple calculations show that one can choose $|n_{k}|\simeq 3^{\lambda n k}$ for suitable $\lambda$ cf. \cite{KW}. For this choice 
\begin{equation}\label{eq: logdeg}
\log \operatorname{deg}F \simeq \log \operatorname{deg} G\simeq  n^2.
\end{equation}
From now on we will assume that
\[
|a_k|\leq 1
\]
Therefore $|b_M|\leq 1$. Then, by the triangle inequality,
\begin{equation}\label{mieszana}
\begin{split} 
\big|z_2\frac{\partial}{\partial z_2}\big(z_1\frac{\partial}{\partial z_1}\big)F({\bf 1})\big|&+\big|z_2\frac{\partial}{\partial z_2}\big(z_1\frac{\partial}{\partial z_1}\big)G({\bf 1})\big|  \\
&>\big|\kappa\sum_{k\in A} a_k{q_{k-1}({\bf 1})}\big|+
\big|\kappa\sum_{k\in A} \bar a_k{p_{k-1}({\bf 1})}\big|\\
&- \sum_{\max M\in A}^n | b_M|\cdot\big|\frac{m(M)_2}{m(M)_1}-\kappa\big|-
\sum_{\max M\notin A}|b_M|\cdot\big|\frac{m(M)_2}{m(M)_1}\big|\\
&\geq  \frac{ \kappa}{2\pi}\sum_{k=1}^n |a_k|\prod_{j=0}^{k-1}(1+|a_j|^2)^{1/2}-2
\end{split}
\end{equation}
Moreover, for $|z_1|=|z_2|=1$,
\begin{equation}\label{czysta2}
\begin{split} 
\big|\big(z_2\frac{\partial}{\partial z_2}\big)^2F(z )\big|+&\big|\big(z_2\frac{\partial}{\partial z_2}\big)^2G(\bar z)\big|  \\
&<\big|\kappa^2\sum_{k\in A} a_kz^{n_k} {q_{k-1}(z)}\big|+
\big|\kappa^2\sum_{k\in A} \bar a_k\bar{z}^{n_k}{p_{k-1}(\bar z)}\big|\\
&+  \sum_{\max M\in A}^n | b_M|\cdot\big|\frac{m(M)_2}{m(M)_1}-\kappa^2\big|+
\sum_{\max M\notin A}|b_M|\cdot\big|\frac{m(M)_2}{m(M)_1}\big|\\
&\leq \sqrt{2} \kappa^2\sum_{k=1}^n |a_k|\prod_{j=0}^{k-1}(1+|a_j|^2)^{1/2}+2
\end{split}
\end{equation}
We put $a_k=k^{-\frac{1}{2}}$ and $\kappa=n^{-\frac{1}{4}}$. Substituting aforementioned values in \eqref{czyste1} gives us
\[
\begin{split}
\max\{\|\big(z_1\frac{\partial}{\partial z_1}\big)^2F\|_{A(D^2)},\|\big(z_1\frac{\partial}{\partial z_1}\big)^2G\|_{A(D^2)}\}&\leq \prod_{k=1}^{n} (1+\frac{1}{k})^{\frac{1}{2}}+1=\prod_{k=1}^{n} (\frac{k+1}{k})^{\frac{1}{2}}+1
\\&=\sqrt{n+1}+1.
\end{split}
\]
By \eqref{czysta2}
\[
\begin{split}
\max\{\|\big(z_2\frac{\partial}{\partial z_2}\big)^2F\|_{A(D^2)},\|\big(z_2\frac{\partial}{\partial z_2}\big)^2G\|_{A(D^2)}\}&\leq \sqrt{2} n^{-\frac{1}{2}} \sum_{k=1}^n k^{-\frac{1}{2}} \prod_{j=0}^{k-1}(1+\frac{1}{j})^{1/2}+2
\\&\leq  \sqrt{2n} +2.
\end{split}
\]
Finally by \eqref{mieszana} for mixed derivative we get
\[
\begin{split}
\max\{\|z_2\frac{\partial}{\partial z_2}\big(z_1\frac{\partial}{\partial z_1}\big)F\|_{A(D^2)},\|z_2\frac{\partial}{\partial z_2}\big(z_1\frac{\partial}{\partial z_1}\big)G\|_{A(D^2)}\}&\geq \frac{n^{\frac{1}{4}}}{ 2\pi }\sum_{k=1}^n k^{-\frac{1}{2}} \prod_{j=0}^{k-1}(1+\frac{1}{j})^{1/2}-2
\\&\geq \frac{n^{\frac{3}{4}}}{ 2\pi}-2.
\end{split}
\]
Hence for fixed $n$ we get in Theorem \ref{dyskbezs}
\[
c= \left(\frac{1}{2\sqrt{2} \pi}+o(1)\right)n^{\frac{1}{4}}
\]
\end{proof}
Considering now either $F(z)$ or $G(z)$ and a sequence $\{n_k\}$ such that $(i)-(iv)$ and \eqref{eq: logdeg} are satisfied we obtain Theorem \ref{thm: stala}.

Now we can deduce Theorem \ref{czystebezs} from the proof of Theorem \ref{dyskbezs}.
\begin{proof}[Proof of Theorem \ref{czystebezs}]
From the construction $(n_{k})_1> 2^{k^2}$. Then 
\[
\begin{split}
|F(z)|+|G(z)|\leq \sum_{M\subset \{1,\ldots,n\}} |b_M| \frac{1}{m(M)^2_1}& \leq \sum_{k=1}^{n}\sum_{\max M=k} \frac{1}{(n_{k})^2_1}
\\&\leq \sum_{k=1}^{n} \frac{2^k}{2^{2k^2}} < 2.
\end{split}
\]
Similarly
\[
\begin{split}
\left|z_1 \frac{\partial}{\partial z_1}F(z)\right|+\left|z_1 \frac{\partial}{\partial z_1} G(z)\right|&\leq \sum_{M\subset \{1,\ldots,n\}} |b_M| \frac{1}{m(M)_1} 
\\& \leq \sum_{k=1}^{n}\sum_{\max M=k} \frac{1}{|(n_{k})_1|}\leq \sum_{k=1}^{n} \frac{2^k}{2^{k^2}} < 2.
\end{split}
\]
Since we may assume $\frac{m(M)_2}{m(M)_1}\leq 1$ we get
\[
\begin{split}
\left|z_2 \frac{\partial}{\partial z_2}F(z)\right|+\left|z_2 \frac{\partial}{\partial z_2} G(z)\right|&\leq \sum_{M\subset \{1,\ldots,n\}} |b_M| \frac{m(M)_2}{m(M)^2_1} 
\\& \leq \sum_{k=1}^{n}\sum_{\max M=k} \frac{1}{|(n_{k})_1|}\leq \sum_{k=1}^{n} \frac{2^k}{2^{k^2}} < 2.
\end{split}
\]
Now since $\frac{\partial}{\partial z_2}F(z)$ is an analytic function we have following identity
\[
\left\|\frac{\partial}{\partial z_2}F(z)\right\|_{\infty}= \sup_{|z_1|=|z_2|=1} \left|\frac{\partial}{\partial z_2}F(z)\right|= \sup_{|z_1|=|z_2|=1} \left|z_1\frac{\partial}{\partial z_2}F(z)\right|< 2.
\]
On the other hand for $j\in\{1,2\}$
\[
z_j\frac{\partial}{\partial z_j} \left(z_j\frac{\partial}{\partial z_j}\right)F(z)=z_j^2\frac{\partial^2}{\partial z_j^2} F(z)+ z_j \frac{\partial}{\partial z_j} F.
\]
Once again because partial derivative is analytic we have
\[
\begin{split}
\left\|\frac{\partial^2}{\partial z_i^2} F(z)\right\|_{\infty}&= \sup_{|z_1|=|z_2|=1} \left|\frac{\partial^2}{\partial z_i^2} F(z)\right|= \sup_{|z_1|=|z_2|=1} \left|z_i^2\frac{\partial^2}{\partial z_i^2} F(z)\right|
\\&\leq  \left\|z_i\frac{\partial}{\partial z_i}\left(z_i\frac{\partial}{\partial z_i}\right)F\right\|_{\infty} + \left\|z_i \frac{\partial}{\partial z_i} F\right\|_{\infty}< 1+2=3.
\end{split}
\]
Finally for mixed derivatives we have 
\[
\begin{split}
\left\|\frac{\partial}{\partial z_1}\left(\frac{\partial}{\partial z_2}\right)F\right\|_{\infty}&=\sup_{|z_1|=|z_2|=1} \left|\frac{\partial}{\partial z_1}\left(\frac{\partial}{\partial z_2}\right)F\right|
\\&=\sup_{|z_1|=|z_2|=1}\left|z_1\frac{\partial}{\partial z_1}\left(z_2\frac{\partial}{\partial z_2}\right)F\right|
\\&=\left\|z_1\frac{\partial}{\partial z_1}\left(z_2\frac{\partial}{\partial z_2}\right)F\right\|_{\infty}.
\end{split}
\]
By analogous arguments we obtain estimates for the function $G$. One of the functions $F$ or $G$ is the wanted polynomial.
\end{proof}

\section{Qualitative approach}\label{sect: qapp}
In this section we give a qualitative proof of Theorem \ref{Torbezs}. It is simpler but  in contrast to Theorem \ref{thm: stala}, does not provide any dependence of the constant on the degree of involved polynomials. It is based, as the proof by Mirkhil and DeLeeuw \cite{Leeuw1964}, on the Hahn-Banach theorem. The new ingredient is the use of Wiener's singularity criterion.

\begin{proof}[Qualitative proof of Theorem \ref{Torbezs}] Denote by $C_{\mathbb{Z}^2_+}^{\{\partial^2_1,\partial_2^2\}}(\mathbb{T}^2)$ the closure of analytic trigonometric polynomials of mean zero with respect to the norm \[
|\!\|f\|\!|=\|D^2_{xx} f  \|_{\infty}+\|D^2_{yy}f \|_{\infty}
\]
This space embeds isometrically as closed subspace into $C(\mathbb{T}^2,\mathbb{C}^2)$ by the formula
\[
f \rightarrow (D^2_{xx}f,D^2_{yy}f) 
\]
Suppose that the assertion of Theorem \ref{Torbezs} does not hold then the functional $\Lambda:C_{\mathbb{Z}^2_+}^{\{\partial^2_1,\partial_2^2\}}(\mathbb{T}^2)\rightarrow \mathbb{C}$ given by the formula 
\[
\Lambda(f)= D^2_{xy} f(0,0) 
\]
is continuous. By the Hahn-Banach theorem it extends to $C(\mathbb{T}^2,\mathbb{C}^2)$. Thus there exists a pair of bounded measures $\mu$ and $\nu$ such that 
\[
  D^2_{xy} f(0,0) =\int_{\mathbb{T}} D^2_{xx}f \operatorname{d} \mu + \int_{\mathbb{T}} D^2_{xx}f \operatorname{d} \mu    
\]
Applying the above formula to characters $(t,s)\rightarrow e^{i(n t+ms)}$ we get
\begin{equation}\label{freq}
\frac{nm}{n^2+m^2} =\frac{n^2}{n^2+m^2}\, \widehat{\mu}(n,m)+\frac{m^2}{n^2+m^2}\,\widehat{\nu}(n,m).
\end{equation}
Let
\[
K_r(n,m)=\{(k,l)\in \mathbb{Z}^2_+ : \max\{|n-k|,|m-l|\}<r\}.
\]
It follows from \eqref{freq} that  
\begin{equation}\label{wspol}
\begin{split}
\widehat{\mu}(k,l)+\widehat{\nu}(k,l)&=1+o(1)\;\qquad\mbox{for } (k,l) \in K_n(n^2,n^2),\\
 \widehat{\mu}(k,l)&=o(1)\qquad \qquad\mbox{for }  (k,l)\in K_n(n^2,n),\\
    \widehat{\nu}(k,l)&=o(1)\qquad\qquad\mbox{for }  (k,l)\in K_n(n,n^2). 
    \end{split}
\end{equation}
Then  
\[
\begin{split}
A&=\liminf_{n\rightarrow\infty}\frac{1}{\# K_n(n^2,n^2)} \sum_{(k,l)\in K_n(n^2,n^2)} |\widehat{\mu}(k,l)|^2+|\widehat{\nu}(k,l)|^2 \geq \frac{1}{4},\\
B&=\lim_{n\rightarrow\infty}\frac{1}{\# K_n(n^2,n)} \sum_{(k,l)\in K_n(n^2,n)} |\widehat{\mu}(k,l)|^2 = 0,\\
C&=\lim_{n\rightarrow\infty}\frac{1}{\# K_n(n,n^2)} \sum_{(k,l)\in K_n(n,n^2)} |\widehat{\nu}(k,l)|^2 = 0.
\end{split}
\]
By Wiener's theorem on the discrete parts of measures \cite[p. 45 ]{Katz}, $A$ is a genuine limit and we get following contradiction
\[
A=\sum_{\tau\in\mathbb{T}^2} |\mu(\{\tau\})|^2+|\nu(\{\tau\})|^2=\sum_{\tau\in\mathbb{T}^2} |\mu(\{\tau\})|^2+\sum_{\tau\in\mathbb{T}^2} |\nu(\{\tau\})|^2=B+C.\\
\]
\end{proof}
In the end we present yet another approach to Theorem \ref{Torbezs}, which could be used to indicate wider class of MDM points.
\begin{prop}
    Let $\Omega\subset \mathbb{C}^2$ and  $(z_0,w_0)\in \partial\Omega$ be such that
    \begin{enumerate}
    \item there are two different supporting hyperplanes $L$, $K$ at $(z_0,w_0)$;
    \item $L\cap K $ is not a complex line; 
    \item $L\cap \overline{\Omega} =\{(z_0,w_0)\}$.
    \end{enumerate}
    Then $(z_0,w_0)$ is an MDM point of $\Omega$.  
\end{prop}
\begin{proof}
    Since we have two supporting hyperplanes at point $(z_0,w_0)$ there are two different  affine functions $\gamma_i:\mathbb{C}^2 \rightarrow \mathbb{C}$  given by
    \[
    \gamma_j(z,w)=x_j(z-z_0) + y_j (w-w_0), \qquad j\in\{1,2\}
    \]
    such that 
    \[
    \operatorname{Re}\left( \gamma_j(z,w)\right) \geq 0 \qquad\mbox{for } (z,w)\in\overline{\Omega}.
    \]

Note that if we have two such functions there is another one $\gamma_3$ with above properties (in fact there are infinitely many). Since $K\cap L$ is not a complex line we may assume that 
\begin{equation}\label{eq: stosunki}
x_j\neq 0,\quad y_j\neq 0,\qquad
\frac{x_i}{y_i}\neq\frac{x_j}{y_j}\qquad\mbox{for }i\neq j,\quad i,j\in\{1,2,3\}.
\end{equation}
Let us consider holomorphic functions 
\[
f_{j,N}(z,w)=N^{-2} e^{-N \gamma_{j(z,w)}}, \qquad j=\{ 1,2,3\}.
\]
We check that 
\[
\begin{split}
\partial_z  f_{j,N}(z,w)&=  -N^{-1}  x_je^{-N\gamma_{j(z,w)}},\quad \partial^2_z  f_{j,N}(z,w)= x_j^2 e^{-N\gamma_{j(z,w)}},\\
\partial_w  f_{j,N}(z,w)&=  -N^{-1}  y_je^{-N\gamma_{j(z,w)}},\quad \partial^2_w f_{j,N}(z,w)= y_j^2 e^{-N\gamma_{j(z,w)}},
\end{split}
\]
and
\[
z\partial_z (w\partial_w) f_{j,N}(z,w)=x_j y_j e^{-N\gamma_{j(z,w)}}.
\]
Therefore we get 
\[
\lim_{N\rightarrow \infty} \partial_z f^N_j(z,w) =\lim_{N\rightarrow \infty} \partial_w f^N_j(z,w)=\lim_{N\rightarrow \infty}  f^N_j(z,w)=0 \qquad \forall (z,w)\in \overline{\Omega} 
\]
and 
\[
\lim_{N\rightarrow \infty} \partial^2_z f^N_j(z,w) = \left\{\begin{array}{cc}
x_j^{2}     & (z,w)=(z_0,w_0);  \\
 0    & (z,w)\neq (z_0,w_0);
\end{array}\right.\mbox{    }\lim_{N\rightarrow \infty} \partial^2_w f^N_j(z,w)= \left\{\begin{array}{cc}
y_j^{2}     & (z,w)=(z_0,w_0);  \\
 0    & (z,w)\neq (z_0,w_0).
\end{array}\right.
\]
Similarly as in the case of bi-disc we obtain with the help of Hahn-Banach theorem, that if $(z_0,w_0)$ is not an MDM point there are bounded Borel measures $\mu_1$, $\mu_2$, $\nu_1$, $\nu_2$, $\nu_3$ such that
\[
\begin{split}
   \partial_z \partial_w f_{j,N}(z_0,w_0)&=\int_{\partial\Omega} \partial_z^2 f_{j,N}(z,w)\operatorname{d} \mu_1 + \int_{\partial\Omega} \partial_w^2 f_{j,N}(z,w)\operatorname{d} \mu_2\\
   \\&\;+\int_{\partial\Omega}\partial_z f_{j,N}(z,w)\operatorname{d} \nu_1 + \int_{\partial\Omega} \partial_w f_{j,N}(z,w)\operatorname{d} \nu_2 + \int_{\partial\Omega}  f_{j,N}(z,w)\operatorname{d} \nu_3. 
\end{split}
\]
Let $\mu_1= \alpha\delta_{(z_0,w_0)}+\tilde{\mu}_1$ and $\mu_2= \beta\delta_{(z_0,w_0)}+\tilde{\mu}_2$, where $\delta_{(z_0,w_0)}\bot \;\tilde{\mu}_1$ and $\delta_{(z_0,w_0)}\bot\; \tilde{\mu}_2$. By Lebesgue's dominated convergence theorem, for $j\in\{1,2,3\}$ we get
\[
\begin{split}
\int_{\partial\Omega} \partial_z^2 f_{j,N}(z,w)\operatorname{d} \mu_1& + \int_{\partial\Omega} \partial_w^2 f_{j,N}(z,w)\operatorname{d} \mu_2\\&= \alpha \partial_z^2 f_{j,N}(z_0,w_0)+ \int_{\partial\Omega} \partial_z^2 f_{j,N}(z,w)\operatorname{d} \tilde{\mu}_1
\\&\;\;+\beta \partial^2_w f_{j,N}(z_0,w_0) + \int_{\partial\Omega} \partial_w^2 f_{j,N}(z,w)\operatorname{d} \tilde{\mu}_2\stackrel{N\to \infty}{\rightarrow} \alpha x_j^2 + \beta y_j^2.
\end{split}
\]
and
\[
\int_{\partial\Omega}\partial_z f_{j,N}(z,w)\operatorname{d} \nu_1 + \int_{\partial\Omega} \partial_w f_{j,N}(z,w)\operatorname{d} \nu_2 + \int_{\partial\Omega}  f_{j,N}(z,w)\operatorname{d} \nu_3\stackrel{N\to \infty}{\rightarrow}0
\]
On the other hand
\[
\lim_{N\rightarrow \infty} \partial_z\partial_w f_{i,N}(z_0,w_0)= x_j y_j.
\]
Hence 
\begin{equation}\label{eq: row}
x_jy_j=\alpha x_j^2+\beta y_j^2\qquad\mbox{for } j\in\{1,2,3\}.
\end{equation}
By \eqref{eq: stosunki} we have $p_j= x_j y_j^{-1}\neq 0$ and 
\begin{equation}\label{systemrow}
 1-\alpha p_j-\beta p_j^{-1}=0, \qquad j\in\{1,2,3\}    
\end{equation}
 Hence $p_1,p_2,p_3$ are roots of the quadratic equation $x -\alpha x^2 -\beta=0$ which stands in contradiction with \eqref{eq: stosunki}. Therefore $(z_0,w_0)$ is an MDM point of $\Omega$.
\end{proof}

\bibliographystyle{plain}
\bibliography{mdm}
Krystian Kazaniecki\\
Institute of Analysis, JKU Linz\\
Institute of Mathematics, University of Warsaw\\
krystian.kazaniecki@jku.at\\
 \\
Micha{\l} Wojciechowski\\
Institute of Mathematics, Polish Academy of Sciences\\
miwoj@impan.pl
\end{document}